\documentclass[a4paper,11pt,reqno]{amsart}
\usepackage{mathrsfs}
\usepackage{cases}
\usepackage{epic}
\usepackage{amsfonts}
\usepackage{graphicx}
\usepackage{amsmath}
\usepackage{amssymb, upgreek}
\usepackage{bm}
\usepackage{latexsym,todonotes}
\usepackage{pdflscape}
\usepackage[all]{xypic}
\usepackage{ytableau}
\usepackage[enableskew,vcentermath]{youngtab}
\usepackage[all]{xy}
\usepackage{color}
\usepackage{colordvi}
\usepackage{multicol}
\usepackage[linktocpage=true]{hyperref}
\usepackage{dsfont}
\hypersetup{colorlinks,linkcolor=blue,urlcolor=cyan,citecolor=blue}
\newcommand{\bvs}{\mathbf{\varsigma}}
\newcommand{\vs}{\varsigma}

\begin{document}
\input xy
\xyoption{all}

\newtheorem{innercustomthm}{{\bf Main~Theorem}}
\newenvironment{customthm}[1]
  {\renewcommand\theinnercustomthm{#1}\innercustomthm}
  {\endinnercustomthm}

  \def\haTh{\widehat{\Theta}}
\def \haH{\widehat{H}}
\newcommand{\bB}{{\mathbf B }}
\newcommand{\bDel}{\boldsymbol{\Delta}}
\newcommand{\bBKH}{\acute{\mathbf H}}
\newcommand{\bH}{\mathbf H}
\newcommand{\BKH}{\acute{H}}
\def \tM{\mathcal{M}\widetilde{\ch}}
\def \E{K}
\def\tTT{\mathrm T}
\def \hath{\widehat{\theta}}
\def \bt{\mathbf t}
\def \bn{\mathbf n}
\def \bh{\mathbf h}
\newcommand{\cc}{{\mathcal C}}
\def \bC{\mathbf C}
\def \dB{\Theta}
\def \bpi{\boldsymbol{\pi}}
\def \dt{{\dot \theta}}
\def \dT{{\dot \Theta}}
\def \co{\mathcal O}
\def \bJ{{\mathbf J}}
\def \ch{{\mathcal H}}
\def \cm{{\mathcal M}}
\def \ct{{\mathcal T}}
\def \bS{\mathbf S}
\def \Uto{\U_{\mathrm{tor}}}
\def \tMH{{\cm\widetilde{\ch}(\Lambda^\imath)}}
\def \tMHl{{\cm\widetilde{\ch}(\bs_\ell\Lambda^\imath)}}
\def \haB{\widehat{B}}
\newcommand{\tMHX}{{}^\imath\widetilde{\ch}(\X_\bfk)}
\newcommand{\tCMH}{{}^\imath\widetilde{\cc}(\bfk Q)}
\newcommand{\tCMHX}{{}^\imath\widetilde{\cc}(\X_\bfk)}
\renewcommand{\mod}{\operatorname{mod}\nolimits}
\newcommand{\tCMHC}{{}^\imath\widetilde{\cc}(\bfk C_n)}
\newcommand{\tCHX}{\widetilde{\cc}(\X_\bfk)}

\numberwithin{equation}{section}

\renewcommand{\ker}{\operatorname{Ker}\nolimits}

\def \tH{{\widetilde{H}}}
\def \tTH{\widetilde{\Theta}}
\def \cp{\mathcal P}
\newcommand{\Hom}{\operatorname{Hom}\nolimits}
\newcommand{\RHom}{\operatorname{RHom}\nolimits}

\newcommand{\Aut}{\operatorname{Aut}\nolimits}
\newcommand{\Id}{\operatorname{Id}\nolimits}
\newcommand{\dimv}{\operatorname{\underline{dim}}\nolimits}
\newcommand{\Sym}{\operatorname{Sym}\nolimits}
\newcommand{\Ext}{\operatorname{Ext}\nolimits}
\newcommand{\Fac}{\operatorname{Fac}\nolimits}
\newcommand{\add}{\operatorname{add}\nolimits}

\def \y{{B}}
\def \haB{\widehat{B}}

\def \bTH { \boldsymbol{\Theta}}
\def \bDel{ \boldsymbol{\Delta}}

\newcommand{\mbf}{\mathbf}
\newcommand{\mbb}{\mathbb}
\newcommand{\mrm}{\mathrm}
\newcommand{\A}{\mathcal A}
\newcommand{\cbinom}[2]{\left\{ \begin{matrix} #1\\#2 \end{matrix} \right\}}
\newcommand{\dvev}[1]{{B_1|}_{\ev}^{{(#1)}}}
\newcommand{\dv}[1]{{B_1|}_{\odd}^{{(#1)}}}
\newcommand{\dvd}[1]{t_{\odd}^{{(#1)}}}
\newcommand{\dvp}[1]{t_{\ev}^{{(#1)}}}
\newcommand{\ev}{\bar{0}}
\newcommand{\End}{\mrm{End}}
\newcommand{\rank}{\mrm{rank}}
\newcommand{\de}{\delta}
\def \C{{\mathbb C}}
\newcommand{\kk}{\widetilde{\mathbf{K}}}
\newcommand{\la}{\lambda}
\newcommand{\LR}[2]{\left\llbracket \begin{matrix} #1\\#2 \end{matrix} \right\rrbracket}
\newcommand{\N}{\mathbb N}
\newcommand{\bbZ}{\mathbb Z}
\newcommand{\odd}{\bar{1}}
\newcommand{\one}{\mathbf 1}
\newcommand{\ov}{\overline}
\newcommand{\qbinom}[2]{\begin{bmatrix} #1\\#2 \end{bmatrix} }
\newcommand{\Q}{\mathbb Q}
\newcommand{\sll}{\mathfrak{sl}}
\newcommand{\T}{\texttt{\rm T}}
\newcommand{\ttt}{\mathfrak{t}}
\newcommand{\U}{\mbf U}
\newcommand{\K}{\mathbb K}
\newcommand{\F}{\mathbb F}
\newcommand{\bi}{\imath}
\newcommand{\bs}{\mathbf s}
\newcommand{\arxiv}[1]{\href{http://arxiv.org/abs/#1}{\tt arXiv:\nolinkurl{#1}}}
\newcommand{\Udot}{\dot{\mbf U}}
\newcommand{\UA}{{}_\A{\mbf U}}
\newcommand{\UAdot}{{}_\A{\dot{\mbf U}}}
\newcommand{\Ui}{{\mbf U}^\imath}
\newcommand{\Uj}{{\mbf U}^\jmath}
\newcommand{\vev}{v^+_{2\la} }
\newcommand{\vodd}{v^+_{2\la+1} }
\newcommand{\Y}{\bB}
\newcommand{\Z}{\mathbb Z}
\newcommand{\Pa}{\operatorname{Pa}\nolimits}
\newcommand{\TT}{\mathbf T}
\newcommand{\B}{\mbf V}
\newcommand{\D}{\mbf D}
\newcommand{\BA}{{}_\A{\B}}
\newcommand{\DA}{{}_\A{\B}'}
\def \X{\mathbb X}
\def \fg{\mathfrak{g}}
\def \bU{{\mathbf U}}
\newcommand{\tK}{\tilde{k}}
\def \I{\mathbb{I}}
\def \bv{v}
\def \cv{\mathcal V}
\def \cu{\mathcal U}
\def \LaC{\Lambda_{\texttt{Can}}}
\newcommand{\tUiD}{{}^{\text{Dr}}\tUi}
\def \nua{a}
\def \hn{\widehat{\mathfrak{n}}}

\newcommand{\UU}{{\mathbf U}\otimes {\mathbf U}}
\newcommand{\UUi}{(\UU)^\imath}
\newcommand{\tUU}{{\tU}\otimes {\tU}}
\newcommand{\tUUi}{(\tUU)^\imath}
\newcommand{\tUi}{\widetilde{{\mathbf U}}^\imath}
\newcommand{\sqq}{{\bf v}}
\newcommand{\sqvs}{\sqrt{\vs}}
\newcommand{\dbl}{\operatorname{dbl}\nolimits}
\newcommand{\swa}{\operatorname{swap}\nolimits}
\newcommand{\Gp}{\operatorname{Gp}\nolimits}
\newcommand{\coker}{\operatorname{Coker}\nolimits}

\newcommand{\tU}{\widetilde{\mathbf U}}

\def \btau{{{\tau}}}
\newcommand{\tk}{\Bbbk}

\def \ff{B}
\def \cJ{\mathcal{J}}

\def \fn{\mathfrak{n}}
\def \fh{\mathfrak{h}}
\def \fu{\mathfrak{u}}
\def \fv{\mathfrak{v}}
\def \fa{\mathfrak{a}}
\def \fk{\mathfrak{k}}

\def \cl{L}

\def \tf{\widetilde{f}}

\def \K{\mathbb{K}}
\def \R{\mathbb{R}}
\def \SS{\mathbb{S}}

\def \BF{\digamma}
\def \BG{\mathbb G}
\def \tMHL{{}^{\imath}\widetilde{\ch}(\LaC^{\imath,op})}
\def \cc{\mathcal C}
\def\ca{\mathcal A}
\newcommand{\Iso}{\operatorname{Iso}\nolimits}
\renewcommand{\Im}{\operatorname{Im}\nolimits}
\newcommand{\res}{\operatorname{res}\nolimits}
\newcommand{\iH}{{}^\imath\widetilde{\ch}}
\newcommand{\Mod}{\operatorname{Mod}\nolimits}
\newcommand{\coh}{\operatorname{coh}\nolimits}
\newcommand{\rep}{\operatorname{rep}\nolimits}
\newcommand{\Ker}{\operatorname{Ker}\nolimits}
\newcommand{\tUiDgr}{{\text{gr}}\tUiD}

\def \G{\mathbb G}
\def \PL{\mathbb{P}^1_{\bfk}}
\def \scrM{\mathscr M}
\def \scrf{\mathscr F}
\def \scrt{\mathscr T}
\def \P{\mathbb P}
\def \cI{\mathcal I}
\def \cJ{\mathcal J}
\def \cs{{\mathcal{S}}}
\def \ch{\mathcal H}
\def \cd{\mathcal D}
\def\bfk{\mathbf{k}}
\def \ck{\mathcal K}
\def \bp{\mathbf p}
\def \ul{\underline}
\def \fp{\mathfrak p}
\def \QJ{Q_{\texttt{J}}}
\def \II{\I_0}
\def \Lg{L\fg}

\def \bvs{{\boldsymbol{\varsigma}}}
\newcommand{\ci}{{\I}_{\btau}}
\def \cn{\mathcal N}

\def\tor{{\rm tor}}
\newcommand{\calc}{{\mathcal C}}
\newcommand{\haT}{\widehat{\Theta}}

\def \La{\Lambda}
\def \iLa{\Lambda^\imath}
\def \BH{\mathbb H}
\def \bH{\mathbf{H}}
\newcommand{\gr}{\operatorname{gr}\nolimits}
\newcommand{\wt}{\text{wt}}
\def \cR{\mathcal R}
\def \ce{\mathcal E}
\def \cb{\mathcal B}
\def \bla{\boldsymbol{\lambda}}
\def \blx{x}

\newtheorem{theorem}{Theorem}[section]
\newtheorem{acknowledgement}[theorem]{Acknowledgement}
\newtheorem{lemma}[theorem]{Lemma}
\newtheorem{proposition}[theorem]{Proposition}
\newtheorem{corollary}[theorem]{Corollary}
{\theoremstyle{remark}
\newtheorem{remark}[theorem]{Remark}
}
{\theoremstyle{definition}
\newtheorem{definition}[theorem]{Definition}
}
\newtheorem*{thm}{Theorem}
\numberwithin{equation}{section}

\title[$\imath$Hopf algebras associated with self-dual Hopf algebras]{$\imath$Hopf algebras associated with self-dual Hopf algebras}

\author[Jiayi Chen]{Jiayi Chen}
\address{ School of Mathematical Sciences, Xiamen University, Xiamen 361005, P.R. China}
\email{jiayichen.xmu@foxmail.com}

\author[Shiquan Ruan]{Shiquan Ruan}
\address{ School of Mathematical Sciences, Xiamen University, Xiamen 361005, P.R. China}
\email{sqruan@xmu.edu.cn}

\makeatletter \@namedef{subjclassname@2020}{\textup{2020} Mathematics Subject Classification} \makeatother


\begin{abstract} 
	Motivated by the construction of $\imath$Hall algebras and $\Delta$-Hall algebras, we introduce $\imath$Hopf algebras associated with  symmetrically self-dual Hopf algebras. We prove that the $\imath$Hopf algebra is an associative algebra with a unit, where the associativity relies on  an analogue of Green's formula in the framework of Hopf algebras. As an application, we construct the $\imath$Taft algebra of dimension 4, which is proved to be isomorphic to the group algebra of $\mathbb{Z}/4\mathbb{Z}$. 
\end{abstract}

\maketitle

 \setcounter{tocdepth}{2}

\section{Introduction}

 Self-dual Hopf algebras arise frequently in various areas of mathematics.  For instance, the group algebra of an abelian
 group over an algebraically closed field is a self-dual Hopf algebra. Additionally, if $A$ is a finite-dimensional Hopf algebra, the tensor product algebra $A\otimes A^*$ is self-dual. 
Among the simplest and most important classes of self-dual Hopf algebras are the Taft algebras, which are neither commutative nor cocommutative, (see \cite{T71}). 
In 2000, Green and Marcos introduced the concept of self-dual (graded) Hopf algebras and provided a significant construction method \cite{GM00}. Later, Huang, Li, and Ye studied self-dual coradically graded pointed Hopf algebras \cite{HLY}, while Ahmed and Li investigated self-dual (semilattice) graded weak Hopf algebras \cite{AL08}. Furthermore, self-dual Hopf algebras can also be defined using posets, a subject that has been deeply investigated in \cite{MR11,Fo}.
Davydov and Runkel applied self-dual Hopf algebras to construct $\mathbb{Z}/2\mathbb{Z}$-graded monoidal categories \cite{DR13}.
These algebras play a significant role in quantum groups, homology theory, and representation theory, demonstrating their broad importance in both pure and applied mathematics.

The Hall algebra approach to quantum groups  has become a prominent topic in representation theory of algebras.
In 1990, Ringel \cite{R1} constructed a Hall algebra associated with a quiver $Q$ over a finite field, and identified its generic version with the positive part of quantum groups when $Q$ is of Dynkin type.
Later Green \cite{G} generalized it to Kac-Moody setting, proving that the Hall algebra is a bialgebra and obtaining a Hopf pairing. In 1997, Xiao \cite{X97} discovered the antipode, leading to the conclusion that the Hall algebra has a self-dual Hopf algebra structure.
In recent years, Lu-Wang \cite{LW19a, LW20a} developed $\imath$Hall algebras of $\imath$quivers to realize universal quasi-split $\imath$quantum groups of Kac-Moody type.
The 
$\imath$Hall algebra provides a convenient basis consisting of isomorphism classes together with the Grothendieck group of a category and a simplified multiplication formula \cite{LW20a,CLR}.
Building on this multiplication formula,  the authors in \cite{CLinR}
 introduced the $\Delta$-Hall algebra associated with a hereditary abelian category. This algebra has the same linear space as the Ringel-Hall algebra, but with structure constants composed of three Hall numbers. The $\Delta$-Hall algebras are closely related to (semi-)derived Hall algebras and $\imath$Hall algebras, offering a new realization of $\imath$quantum groups of split type.

With the goal of realizing arbitrary 
$\imath$quantum groups, we aim to extend the construction of 
$\imath$Hall algebras and 
$\Delta$-Hall algebras within the framework of self-dual Hopf algebras. In this paper, we introduce the concept of 
$\imath$Hopf algebras associated with finite-dimensional self-dual Hopf algebras. In a forthcoming paper, we will extend this to the infinite-dimensional case and use 
$\imath$Hopf algebras to realize arbitrary 
$\imath$quantum groups.

More precisely,
let $A$ be a finite-dimensional bialgebra (Hopf algebra) with basis $\{c_i\}_{i\in I}$, where the multiplication and comultiplication are given by  the following formulas for $i,j,k\in I$:
\[
c_i\cdot c_j=\sum_{k\in I} F_{ij}^k c_k, \quad \Delta(c_k)=\sum_{i,j\in I} G_{ij}^k c_i\otimes c_j.
\]
When the structure constants satisfy $F_{ij}^k=G_{ij}^k$ for any $i,j,k\in I$, we define a new algebra structure on $A$ with multiplication 
\[
c_i\diamond c_j=\sum_{k\in I}  {}^\imath F_{ij}^k c_k, 
\]
where the structure constants
\[
  {}^\imath F^k_{ij}=\sum_{i',j',k'}F^i_{k'j'} F^j_{i'k'} F^k_{j'i'},\qquad\text{  for any  }i,j,k\in I.
 \]
The associativity of the new multiplication relies on an analogue of Green's formula in the framework of bialgebras:
 \begin{align*}
       \sum_{k}F_{ij}^k G_{k'k''}^k=\sum_{i',i'',j',j''}G_{i'i''}^iG_{j'j''}^j F_{i'j'}^{k'} F_{i''j''}^{k''}.
  \end{align*}

If the structure constants satisfy $F_{ij}^k=G_{ij}^k\cdot\dfrac{a_k}{a_ia_j}$ for any $i,j,k\in I$, where $a_i$ is a nonzero element in the field $\bfk$ for each $i\in I$, we can also obtain an algebra structure on $A$ with the modified structure constants 
\[
   {}^\imath F^k_{ij}=\sum_{i',j',k'}G^i_{k'j'} G^j_{i'k'} F^k_{j'i'}\cdot\dfrac{1}{a_{j'}},  \qquad\text{  for any  }i,j,k\in I.
   \]


The main purpose of this paper is to construct the $\imath$Hopf algebras associated with self-dual Hopf algebras. In this case, the structure constants do not behave well in general, so we need to use the base-change strategy to reach a simpler case. 
For this approach, we will require the translation formula for the structure constants under the base-change procedure.
Namely,
assume $\{d_{\tilde{i}}\}_{\tilde{i}\in I}$ is another basis of $A$ and $\widetilde{F}_{\tilde{i}\tilde{j}}^{\tilde{k}}, \widetilde{G}_{\tilde{i}\tilde{j}}^{\tilde{k}}$ are  structure constants under this basis. 
Denote the transition matrix from $\{c_i\}_{i\in I}$ to $\{d_{\tilde{i}}\}_{\tilde{i}\in I}$ by $T=(t_{i\tilde{i}})$ and its inverse by $T^{-1}=(\hat{t}_{\tilde{i}i})$. 
We obtain
\[
\widetilde{F}_{\tilde{i}\tilde{j}}^{\tilde{k}}=\sum_{i,j,k}t_{i\tilde{i}}t_{j\tilde{j}}\hat{t}_{\tilde{k}k}F_{ij}^k, \qquad\widetilde{G}_{\tilde{i}\tilde{j}}^{\tilde{k}}=\sum_{i,j,k}\hat{t}_{\tilde{i}i}\hat{t}_{\tilde{j}j}{t}_{k\tilde{k}}G_{ij}^k.
 \]
Finally, attached to a symmetrically Hopf algebra $A$ with self-dual $\varphi:A\to A^*$, we construct an $\imath$Hopf algebra $A^\imath$, whose structure constant is defined by
 \[
 {}^\imath \widetilde{F}^{\tilde{k}}_{\tilde{i}\tilde{j}}=\sum_{y_1,y_2,x,z}
    s_{y_1y_2}\widetilde{G}_{y_1x}^{\tilde{i}}\widetilde{G}_{zy_2}^{\tilde{j}}\widetilde{F}_{xz}^{\tilde{k}} ,  \qquad\text{  for any  }\tilde{i},\tilde{j},\tilde{k}\in I,
 \]
 where $(s_{ij})$ is a representation matrix of $\varphi$. As an application, we obtain the associated $\imath$Taft algebra of dimension 4, and prove that it is commutative.

The paper is organized as follows. In Section 2, we formulate the formulas of structure constants for a Hopf algebra. Two simple cases of $\imath$Hopf algebras are considered in Section 3, and we prove the associativity of $\imath$Hopf algebras in these cases. Section 4 is devoted to describing how the structure constants change under the base-change procedure, then we define the general $\imath$Hopf algebras associated with symmetrically self-dual Hopf algebras.
Section 5 is an application to the Taft algebras.

\section{Formulas on structure constants for a Hopf algebra}

Let $\bfk$ be an arbitrary field, and $A$ be a $\bfk$-linear space of finite dimension. Assume there is a basis $\{c_i\}_{i\in I}$ of $A$. If $A$ has an algebra structure $(A,\cdot,u)$, we can describe it by structure constants $F_{ij}^k$ and $\lambda_i$, that is, for any $i,j,k\in I$ we have
\begin{align}\label{algebra structure}
 c_i\cdot c_j=\sum_{k\in I} F_{ij}^k c_k,\qquad u(1)=\sum_{i}\lambda_i c_i.   
\end{align}
Similarly, if $A$ has a coalgebra structure $(A,\Delta,\varepsilon)$, we can also describe it by structure constants $G_{ij}^k$ and $\mu_i$ for $i,j,k\in I$, with the relations
\begin{align}\label{coalgebra structure}
\Delta(c_k)=\sum_{i,j\in I} G_{ij}^k c_i\otimes c_j,\qquad \varepsilon(c_i)=\mu_i.
\end{align}
The following proposition describes when $A$ is a bialgebra, via the equalities on constant structures.

\begin{proposition}
\label{green}
  Let $A$ be a vector space with an algebra structure $(A,\cdot,u)$ and a coalgebra structure $(A,\Delta,\varepsilon)$. 
  
  $(1)$ The comultiplication $\Delta$ is an algebra morphism if and only if for any $i,j,k',k''$, we have 
  \begin{align}
  \label{bialg1}
       \sum_{k}F_{ij}^k G_{k'k''}^k=\sum_{i',i'',j',j''}G_{i'i''}^iG_{j'j''}^j F_{i'j'}^{k'} F_{i''j''}^{k''};
  \end{align}

    $(2)$ The counit $\varepsilon$ is an algebra morphism if and only if for any $i,j$, the following holds
    \begin{align}
    \label{bialg2}
        \sum_{k\in I} F_{ij}^k \mu_k=\mu_i\mu_j;
    \end{align}

     $(3)$ The unit $u$ is a coalgebra morphism if and only if for any $i,j$, the following holds
    \begin{align}
    \label{bialg3}
        \sum_{k\in I} G_{ij}^k \lambda_k=\lambda_i\lambda_j.
    \end{align}
\end{proposition}

\begin{proof}
   $(1)$ The condition $\Delta$ is an algebra morphism gives\[
    \Delta(c_i\cdot c_j)=\Delta(c_i)\cdot \Delta(c_j), \text{ for any }i,j\in I.
    \]The left-hand side is equal to
    \begin{align*}
        \Delta(c_i\cdot c_j)=\Delta(\sum_{k}F_{ij}^k c_k)=\sum_{k}F_{ij}^k \cdot\Delta(c_k)=\sum_{k}F_{ij}^k \sum_{k',k''}G_{k'k''}^kc_{k'}\otimes c_{k''},
    \end{align*}
    while the right-hand side is equal to 
    \begin{align*}
        \Delta(c_i)\cdot \Delta(c_j)&=\sum_{i',i''} G_{i'i''}^i c_{i'}\otimes c_{i''} \cdot \sum_{j',j''} G_{j'j''}^j c_{j'}\otimes c_{j''} 
        \\&=\sum_{i',i'',j',j''}G_{i'i''}^iG_{j'j''}^j(c_{i'}\cdot c_{j'})\otimes(c_{i''}\cdot c_{j''})
        \\&=\sum_{i',i'',j',j''}G_{i'i''}^iG_{j'j''}^j(\sum_{k'}F_{i'j'}^{k'}c_{k'})\otimes(\sum_{k'}F_{i''j''}^{k''}c_{k''})
    \\&=\sum_{i',i'',j',j'',k',k''}G_{i'i''}^iG_{j'j''}^j F_{i'j'}^{k'} F_{i''j''}^{k''}\cdot c_{k'}\otimes c_{k''}.
    \end{align*}
   Then  \eqref{bialg1} follows by comparing the coefficients of $c_{k'}\otimes c_{k''}$ on both sides. 

    $(2)$ The condition $\varepsilon$ is an algebra morphism gives\[
    \varepsilon(c_i\cdot c_j)=\varepsilon(c_i)\cdot \varepsilon(c_j), \text{ for any }i,j\in I.
    \]
    The left-hand side is equal to \[
    \varepsilon(c_i\cdot c_j)=\varepsilon(\sum_{k\in I} F_{ij}^k c_k)=\sum_{k\in I} F_{ij}^k \mu_k,
    \]
    while the right-hand side  is equal to\[
    \varepsilon(c_i)\cdot \varepsilon(c_j)=\mu_i\mu_j.
    \]
    Therefore, to make  both sides equal, we must have\[
    \sum_{k\in I} F_{ij}^k \mu_k=\mu_i\mu_j.
    \]

    $(3)$ The proof is entirely similar to $(2)$.
\end{proof}

This equation \eqref{bialg1} can be viewed as a generalized version of Green's formula in the context of bialgebras. We shall understand this formula by the following diagram:
\[
    \begin{tikzpicture}
	\node (0) at (0,0) {$c_{j'}$};
	\node (1) at (2-.5,0) {$c_{k'}$};
	\node (2) at (4-1,0){$c_{i'}$};
	\node (01) at (0,2-.5) {$c_j$};
	\node (11) at (2-.5,2-.5) {$c_k$};
	\node (21) at (4-1,2-.5){$c_i$};
	\node (02) at (0,4-1) {$c_{j''}$};
	\node (12) at (2-.5,4-1) {$c_{k''}$};
	\node (22) at (4-1,4-1){$c_{i''}$};
	
	\draw[->] (01) --node[above ]{} (11);
	\draw[->] (11) --node[above ]{} (21);
	\draw[->][dashed] (12) --node[above ]{} (11);
	\draw[->][dashed] (11) --node[above ]{} (1);
	
	\draw[->] (0) --node[above ]{} (1);
	\draw[->] (1) --node[above ]{} (2);	
 \draw[->] (02) --node[above ]{} (12);
	\draw[->] (12) --node[above ]{} (22);
		\draw[<-][dashed] (0) --node[above ]{} (01);
	\draw[<-][dashed] (01) --node[above ]{} (02);
		\draw[->][dashed] (21) --node[above ]{} (2);
	\draw[->][dashed] (22) --node[above ]{} (21);	\end{tikzpicture}
    \]
 where $F_{ij}^k$ presents the row sequence with solid arrows $c_j\longrightarrow  c_k\longrightarrow c_i$, 
while $G_{k'k''}^k$ presents the (middle) column sequence with dotted arrows. 
Roughly speaking, the left-hand side of \eqref{bialg1} sums over all the crosses inside the diagram, whereas the right-hand side of \eqref{bialg1} sums over all the squares outside the diagram.

\begin{remark}
    Let $A$ be a $\bfk$-linear space with a basis $\{c_i\}_{i\in I}$ of $A$. Then the algebra structure on $A$ can be described via structure constants  as follows $(\forall i,j,k,l\in I)$:
    \begin{align}
    \label{alg1}
        \sum\limits_{s}F_{ij}^sF_{sk}^l=\sum\limits_{t}F_{it}^lF_{jk}^t:=F^l_{ijk}
    \end{align}
    \begin{align}
    \label{alg2}
        \sum_i \lambda_i F_{ij}^k=\delta_{jk},\qquad
         \sum_j \lambda_j F_{ij}^k=\delta_{ik};
    \end{align}
    and the coalgebra structure  on $A$ can be described by:
    \begin{align}
    \label{coalg1}
        \sum\limits_{s}G_{ij}^sG_{sk}^l=\sum\limits_{t}G_{it}^lG_{jk}^t:=G_{ijk}^l,
    \end{align}
    \begin{align}
    \label{coalg2}
        \sum_i \mu_i G_{ij}^k=\delta_{jk}, \qquad
         \sum_j \mu_j G_{ij}^k=\delta_{ik}.
    \end{align}
\end{remark}


\section{$\imath$Hopf algebras in some simple cases}

Motivated by the construction of $\imath$Hall algebras in \cite{LW19a} and $\Delta$-Hall algebras in \cite{CLinR}, we hope to introduce a new algebra structure, called $\imath$Hopf algebras, within the framework of Hopf algebras. 

We always assume that $A$ is a bialgebra (or Hopf algebra) with a basis $\{c_i\}_{i\in I}$ and the relations \eqref{algebra structure} and \eqref{coalgebra structure} throughout the paper. 
This section is devoted to dealing with the simple
cases where the structure constants $F_{ij}^k$ and $G_{ij}^k$ are closely related to each other, and in the next section we will consider general cases.

\subsection{The case $F_{ij}^k=G_{ij}^k$}

Firstly, let us construct a new algebra structure on the bialgebra $A$ for the most simple case, where the structure constants $F_{ij}^k$ and $G_{ij}^k$ coincide for any $i,j,k\in I$. The proof of associativity relies on the analogue of Green's formula given in \eqref{bialg1}.

\begin{theorem}
\label{thm1}
   Assume  $F_{ij}^k=G_{ij}^k$ holds in $A$ for any $i,j,k\in I$. Let $(A^\imath,\diamond)$ be a vector space with the same basis $\{c_i\}_{i\in I}$ as $A$, equipped with the multiplication $c_i\diamond c_j=\sum_{k\in I} {}^\imath F^k_{ij} c_k$, where \[
   {}^\imath F^k_{ij}=\sum_{i',j',k'}F^j_{i'k'} F^i_{k'j'} F^k_{j'i'}  \qquad\text{  for any  }i,j,k\in I.
   \] Then $(A^\imath,\diamond)$ is  an associative algebra with the unit $u(1)=\sum_{i}\lambda_i c_i$.
\end{theorem}

\begin{proof}
    First, we need to prove 
    \begin{align}
        \label{asso}
        \sum\limits_{s}{}^\imath F_{jk}^s{}^\imath F_{is}^l=\sum\limits_{t}{}^\imath F_{tk}^l{}^\imath F_{ij}^t\qquad\text{  for any  }i,j,k,l \in I.
    \end{align}

    The left-hand side is equal to 
    \begin{align*}
        \sum_{s}{}^\imath F_{jk}^s{}^\imath F_{is}^l
        &=\sum_s\sum_{i',j',s'}F^k_{i'j'} F^j_{j's'} F^s_{s'i'} \sum_{s'',k,'l'}F^s_{s''k'} F^i_{k'l'} F^l_{l's''}
        \\&=\sum_{i',j',s'} \sum_{s'',k',l'} F^k_{i'j'} F^j_{j's'} F^i_{k'l'} F^l_{l's''} \sum_s F^s_{s'i'} F^s_{s''k'}.
    \end{align*}
    By Proposition \ref{green}, we have\[
    \sum_s F^s_{s'i'} F^s_{s''k'}=\sum_{s_1,s_2,s_3,s_4} F^{s'}_{s_1 s_2} F^{s''}_{s_1s_3} F^{k'}_{s_2s_4} F^{i'}_{s_3s_4}.
    \]
    Hence \begin{align}
        \sum_{s}{}^\imath F_{jk}^s{}^\imath F_{is}^l
        &=\sum_{i',j',s'} \sum_{s'',k',l'} F^k_{i'j'} F^j_{j's'} F^i_{k'l'} F^l_{l's''} \sum_{s_1,s_2,s_3,s_4} F^{s'}_{s_1 s_2} F^{s''}_{s_1s_3} F^{k'}_{s_2s_4} F^{i'}_{s_3s_4}
        \notag\\&=\sum_{j',l'}\sum_{s_1,s_2,s_3,s_4}\sum_{i'} F^k_{i'j'}F^{i'}_{s_3s_4}
        \sum_{s'}F^j_{j's'}F^{s'}_{s_1 s_2}
        \sum_{k'}F^i_{k'l'}F^{k'}_{s_2s_4}
        \sum_{s''}F^l_{l's''}F^{s''}_{s_1s_3} 
        \notag\\&=\sum_{j',l'}\sum_{s_1,s_2,s_3,s_4}F^k_{s_3s_4j'}
        F^j_{j's_1 s_2}
        F^i_{s_2s_4l'}
        F^l_{l's_1s_3}\label{long left one}. 
    \end{align}
    
    Similarly, the right-hand side of \eqref{asso} is equal to
    \begin{align}
        \sum_{t}{}^\imath F_{tk}^l{}^\imath F_{ij}^t
        &=\sum_t\sum_{i',t',l'}F^k_{i't'} F^t_{t'l'} F^l_{l'i'} \sum_{j',k',t''}F^j_{j'k'} F^i_{k't''} F^t_{t''j'}
        \notag \\&=\sum_{i',t',l'}\sum_{j',k',t''}F^k_{i't'}  F^l_{l'i'} F^j_{j'k'} F^i_{k't''} \sum_t F^t_{t'l'}F^t_{t''j'}
        \notag \\&=\sum_{i',t',l'}\sum_{j',k',t''}F^k_{i't'}  F^l_{l'i'} F^j_{j'k'} F^i_{k't''}
        \sum_{t_1,t_2,t_3,t_4} F_{t_1t_2}^{t'}  F^{t''}_{t_1t_3} F^{j'}_{t_2t_4} F^{l'}_{t_3t_4}
         \notag \\&=\sum_{i',k'}\sum_{t_1,t_2,t_3,t_4}\sum_{t'}F^k_{i't'}F_{t_1t_2}^{t'}
         \sum_{l'}F^l_{l'i'}F^{l'}_{t_3t_4}
         \sum_{j'}F^j_{j'k'}F^{j'}_{t_2t_4}
         \sum_{t''}F^i_{k't''}F^{t''}_{t_1t_3}
            \notag 
            \\ \label{long right one} &=\sum_{i',k'}\sum_{t_1,t_2,t_3,t_4}F^k_{i't_1t_2}F^l_{t_3t_4i'}
         F^j_{t_2t_4k'}F^i_{k't_1t_3}.
    \end{align}
    Note that the right-hand sides of \eqref{long left one} and \eqref{long right one} are equal when 
    identifing the variables $j',l',s_1$, $s_2$ ,$s_3$, $s_4$ with $t_2,t_3,t_4,k',i',t_1$, respectively.
    Therefore, we obtain the equation \eqref{asso}.

    It remains to prove that $u(1)$ is also the unit for $A^\imath$, which by \eqref{alg2} is equivalent to show
    \[
    \sum_j \lambda_j {}^\imath F_{ij}^k=\delta_{ik}, \qquad
         \sum_i \lambda_i {}^\imath F_{ij}^k=\delta_{jk}.
    \]
    Note that
    \begin{align*}
        \sum_j \lambda_j {}^\imath F_{ij}^k&=\sum_j \lambda_j \sum_{i',j',k'}F^j_{i'j'} F^i_{j'k'} F^k_{k'i'}
        \\&= \sum_{i',j',k'} F^i_{j'k'} F^k_{k'i'} \sum_j \lambda_j G^j_{i'j'}
        \\&= \sum_{i',j',k'} F^i_{j'k'} F^k_{k'i'} {\lambda_{i'}\lambda_{j'}},
    \end{align*}
    where we use the equation \eqref{bialg2}. Then by formula \eqref{alg2}, the above is equal to\[
     \sum_{i',j',k'} F^i_{j'k'} F^k_{k'i'} {\lambda_{i'}\lambda_{j'}}=\sum_{k'}\delta_{kk'}\delta_{ik'}=\delta_{ik}.
    \] Similarly, one can also obtain $ \sum_i \lambda_i {}^\imath F_{ij}^k=\delta_{jk}$. 
    These finish the proof.
\end{proof}

\subsection{The case $F_{ij}^k=G_{ij}^k\cdot\dfrac{a_k}{a_ia_j}$}

In a slightly more complicated case, the structure constant 
$F_{ij}^j$ is not exactly equal to $G_{ij}^j$, but is related to it by certain scalars. In this case, we will adjust the structure constants of the new multiplication.

\begin{proposition}
\label{ialgcor}
    Assume $F_{ij}^k=G_{ij}^k\cdot\dfrac{a_k}{a_ia_j}$ for any $i,j,k\in I$, where each $a_i(i\in I)$ is a nonzero element in $\mathbf{k}$. Let $(A^\imath,\diamond)$ be a vector space with the same basis $\{c_i\}_{i\in I}$ as $A$, equipped with the multiplication $c_i\diamond c_j=\sum_{k\in I} {}^\imath F^k_{ij} c_k$, where \[
   {}^\imath F^k_{ij}=\sum_{i',j',k'}G^j_{i'k'} G^i_{k'j'} F^k_{j'i'}\cdot\dfrac{1}{a_{j'}}  \qquad\text{  for any  }i,j,k\in I.
   \]
   Then $(A^\imath,\diamond)$ is also an associative algebra with the unit $u(1)$.
\end{proposition}

\begin{proof}
    We have
    \begin{align*}
        \sum_{s}{}^\imath F_{jk}^s{}^\imath F_{is}^l
        &=\sum_s\sum_{i',j',s'}G^k_{i'j'} G^j_{j's'} F^s_{s'i'}\dfrac{1}{a_{j'}} \sum_{s'',k,'l'}G^s_{s''k'} G^i_{k'l'} F^l_{l's''}\dfrac{1}{a_{k'}}
        \\&=\sum_s\sum_{i',j',s'}F^k_{i'j'} F^j_{j's'} F^s_{s'i'}\dfrac{a_{i'}a_{j'}a_{s'}}{a_ka_j} \sum_{s'',k,'l'}F^s_{s''k'} F^i_{k'l'} F^l_{l's''}\dfrac{a_{s''}a_{k'}a_{l'}}{a_sa_i}
        \\&=\sum_{i',j',s'} \sum_{s'',k',l'} F^k_{i'j'} F^j_{j's'} F^i_{k'l'} F^l_{l's''} \sum_s F^s_{s'i'} F^s_{s''k'}\dfrac{a_{i'}a_{j'}a_{s'}}{a_ka_j}\dfrac{a_{s''}a_{k'}a_{l'}}{a_sa_i}.
    \end{align*}
By Proposition \ref{green},
    \[
     \sum_s F^s_{s'i'} F^s_{s''k'}\dfrac{a_{s''}a_{k'}}{a_s}=\sum_{s_1,s_2,s_3,s_4} F^{s'}_{s_1 s_2} F^{s''}_{s_1s_3} F^{k'}_{s_2s_4} F^{i'}_{s_3s_4}\dfrac{a_{s_1}a_{s_2}a_{s_3}a_{s_4}}{a_{s'}a_{i'}}.
    \]
    Hence
    \begin{align*}
        \sum_{s}{}^\imath F_{jk}^s{}^\imath F_{is}^l
        &=\sum_{i',j',s'} \sum_{s'',k',l'} F^k_{i'j'} F^j_{j's'} F^i_{k'l'} F^l_{l's''} \dfrac{a_{i'}a_{j'}a_{s'}a_{s''}a_{k'}a_{l'}}{a_ia_ja_k}\sum_s F^s_{s'i'} F^s_{s''k'}\dfrac{1}{a_s}
        \\&=\sum_{i',j',s'} \sum_{s'',k',l'} F^k_{i'j'} F^j_{j's'} F^i_{k'l'} F^l_{l's''}
        \dfrac{a_{i'}a_{j'}a_{s'}a_{s''}a_{k'}a_{l'}}{a_ia_ja_k}
        \\&\qquad\cdot\sum_{s_1,s_2,s_3,s_4} F^{s'}_{s_1 s_2} F^{s''}_{s_1s_3} F^{k'}_{s_2s_4} F^{i'}_{s_3s_4}
   \dfrac{a_{s_1}a_{s_2}a_{s_3}a_{s_4}}{a_{s'}a_{s''}a_{k'}a_{i'}}
   \\&=\sum_{j',l'}\sum_{s_1,s_2,s_3,s_4}F^k_{s_3s_4j'}
        F^j_{j's_1 s_2}
        F^i_{s_2s_4l'}
        F^l_{l's_1s_3}\dfrac{a_{j'}a_{l'}a_{s_1}a_{s_2}a_{s_3}a_{s_4}}{a_ia_ja_k}.
    \end{align*}
Similarly,
\begin{align*}
        \sum_{t}{}^\imath F_{tk}^l{}^\imath F_{ij}^t
        =\sum_{i',k'}\sum_{t_1,t_2,t_3,t_4}F^k_{i't_1t_2}F^l_{t_3t_4i'}
         F^j_{t_2t_4k'}F^i_{k't_1t_3}\dfrac{a_{i'}a_{k'}a_{t_1}a_{t_2}a_{t_3}a_{s_4}}{a_ia_ja_k}.
    \end{align*}
With variables $j',l',s_1,s_2,s_3,s_4$ corresponding to $t_2,t_3,t_4,k',i',t_1$, respectively, it follows that $\sum\limits_{s}{}^\imath F_{jk}^s{}^\imath F_{is}^l=\sum\limits_{t}{}^\imath F_{tk}^l{}^\imath F_{ij}^t$. 

    Finally, by a similar argument to Theorem \ref{thm1}, we can also obtain the unit $u(1)$.
\end{proof}

This proposition recovers Theorem \ref{thm1} which is the special case when $a_i=1$ for all $i\in I$.


\section{$\imath$Hopf algebras and self-dual Hopf algebras}

In this section, we aim to construct a new algebraic structure, called $\imath$Hopf algebras, within the framework of self-dual Hopf algebras.


\subsection{Self-dual Hopf algebra}
Let's first review the definition of self-dual Hopf algebras. Let $A$ be a finite-dimensional Hopf algebra. 
To distinguish from Section 3, we assume that $\{d_{\tilde{i}}\}_{\tilde{i}\in I}$ is a basis of $A$, and
\begin{align}\label{comult of di}
d_{\tilde{i}}\cdot d_{\tilde{j}}=\sum_{\tilde{k}\in I}\widetilde{F}_{\tilde{i}\tilde{j}}^{\tilde{k}}d_{\tilde{k}},\quad\Delta(d_{\tilde{k}})=\sum_{\tilde{i},\tilde{j}\in I} \widetilde{G}_{\tilde{i}\tilde{j}}^{\tilde{k}} d_{\tilde{i}}\otimes d_{\tilde{j}}.
\end{align}
Let $A^*$ be the dual space of $A$ with a dual basis $\{{{d}}^*_{\tilde{i}}\}_{\tilde{i}\in I}$, then there exists a natural bialgebra structure on $A^*$ in the following sense: 
\begin{align}\label{comult of di}
d^*_{\tilde{i}}\cdot d^*_{\tilde{j}}=\sum_{\tilde{k}\in I}\widetilde{G}_{\tilde{i}\tilde{j}}^{\tilde{k}}d^*_{\tilde{k}},\quad\Delta(d^*_{\tilde{k}})=\sum_{\tilde{i},\tilde{j}\in I} \widetilde{F}_{\tilde{i}\tilde{j}}^{\tilde{k}} d^*_{\tilde{i}}\otimes d^*_{\tilde{j}}.
\end{align}

A Hopf algebra $A$ is called \textit{self-dual} if there exists a Hopf algebra isomorphism $\varphi:A\rightarrow A^*$. It is further called \textit{symmetrically self-dual} if $\varphi=\varphi^*$. Here,
$\varphi^*$ is the dual map of $\varphi$ given by
\begin{align}\label{dual map} \varphi^*(d_{\tilde{i}})(d_{\tilde{j}})=\varphi(d_{\tilde{j}})(d_{\tilde{i}}) \text{\quad for\; any\;} \tilde{i}, \tilde{j}\in I.
\end{align}

Let $S=(s_{\tilde{j}\tilde{i}})$ be the matrix representation of $\varphi$ with respect to the bases  $\{d_{\tilde{i}}\}_{\tilde{i}\in I}$ and   $\{d^*_{\tilde{i}}\}_{\tilde{i}\in I}$, namely, $\varphi(d_{\tilde{i}})=\sum_{\tilde{j}}s_{\tilde{j}\tilde{i}}d^*_{\tilde{j}}$. Then it is easy to see that $\varphi=\varphi^*$ if and only if $S$ is symmetric.  




The following result states that under the condition 
 $\varphi=\varphi^*$, in order to prove $\varphi$ is an isomorphism of bialgebras, it suffices to show that it is an isomorphism of algebras or coalgebras.

\begin{proposition}
\label{bialgiso}
    Assume that $A$ is a bialgebra, $A^*$ is the dual bialgebra of $A$, and $\varphi$ is a map between $A$ and $A^*$ satisfying $\varphi=\varphi^*$. Then the following are equivalent:

    $(a)$ The map $\varphi$ is an isomorphism of algebras;

    $(b)$ The map $\varphi$ is an isomorphism of coalgebras.


\end{proposition}

\begin{proof}
    We only prove that $(a)$ implies $(b)$, since the other one is similar. If $\varphi$ is an isomorphism of algebras, then we have a commutative diagram
    \[
    \begin{tikzpicture}
	\node (0) at (0,0) {$A^*\otimes A^*$};
	\node (1) at (2,0) {$A^*$.};
	
	\node (01) at (0,2) {$A\otimes A$};
	\node (11) at (2,2) {$A$};

	\draw[->] (01) --node[above ]{$m$} (11);

	\draw[->] (11) --node[right]{$\varphi$} (1);
	
	\draw[->] (0) --node[above ]{$m^*$} (1);
	
		\draw[<-] (0) --node[left ]{$\varphi\otimes \varphi$} (01);

	\end{tikzpicture}
    \]
    Take the dual maps for the whole diagram, and we obtain
    \[
    \begin{tikzpicture}
	\node (0) at (0,0) {$A\otimes A$};
	\node (1) at (2,0) {$A$};
	
	\node (01) at (0,2) {$A^*\otimes A^*$};
	\node (11) at (2,2) {$A^*$};

	\draw[<-] (01) --node[above ]{$\Delta^*$} (11);

	\draw[<-] (11) --node[right]{$\varphi^*$} (1);
	
	\draw[<-] (0) --node[above ]{$\Delta$} (1);
	
		\draw[->] (0) --node[left ]{$\varphi^*\otimes \varphi^*$} (01);

	\end{tikzpicture}
    \]
    which implies $\varphi^*=\varphi$ is also an isomorphism of coalgebras.
\end{proof}

 \subsection{Base-change procedure}

We aim to define a new algebraic structure on a symmetrically self-dual Hopf algebra $A$, as discussed in the previous section. Note that in general, the equality 
$\widetilde{F}_{\tilde{i}\tilde{j}}^{\tilde{k}} =\widetilde{G}_{\tilde{i}\tilde{j}}^{\tilde{k}}$
  does not hold. To address this, we can apply the base-change strategy to reduce the problem to a simpler case. For this approach, we will require the translation formula for the structure constants under the base-change procedure.

  
  Let $\{c_i\}_{i\in I}$ be a basis of $A$ with 
\begin{align}\label{Fijk and Gijk for ci}
 c_i\cdot c_j=\sum_{k\in I} F_{ij}^k c_k,\qquad \Delta(c_k)=\sum_{i,j\in I} G_{ij}^k c_i\otimes c_j. 
\end{align}
Assume $T=(t_{i\tilde{i}})$ is the transition matrix from the basis $\{c_i\}_{i\in I}$  to  $\{d_{\tilde{i}}\}_{\tilde{i}\in I}$ , i.e.,  $d_{\tilde{i}}=\sum_{i}t_{i\tilde{i}}c_{i}$ for any $\tilde{i}\in I$. Denote by ${T}^{-1}:=(\hat{t}_{\tilde{i}i})$ the inverse matrix of $T$. 
 
\begin{lemma} 
   Keep notations as above. We have
    \begin{align}
    \label{trans1}
\widetilde{F}_{\tilde{i}\tilde{j}}^{\tilde{k}}=\sum_{i,j,k}t_{i\tilde{i}}t_{j\tilde{j}}\hat{t}_{\tilde{k}k}F_{ij}^k, \qquad \widetilde{G}_{\tilde{i}\tilde{j}}^{\tilde{k}}=\sum_{i,j,k}\hat{t}_{\tilde{i}i}\hat{t}_{\tilde{j}j}{t}_{k\tilde{k}}G_{ij}^k.
    \end{align}
\end{lemma}


\begin{proof}
   Reversing the transformation formula $d_{\tilde{i}}=\sum_{i}t_{i\tilde{i}}c_{i}$, we get $c_{{i}}=\sum_{\tilde{i}}\hat{t}_{\tilde{i}i}d_{\tilde{i}}$. By applying this formula to $c_i,c_j,c_k$ in the multiplication $c_i\cdot c_j=\sum_{k\in I} F_{ij}^k c_k$, we obtain \begin{align*}
        \sum_{\tilde{i}}\hat{t}_{\tilde{i}i}d_{\tilde{i}}\cdot \sum_{\tilde{j}}\hat{t}_{\tilde{j}j}d_{\tilde{j}}=\sum_{k} F_{ij}^k \sum_{\tilde{k}}\hat{t}_{\tilde{k}k}d_{\tilde{k}}.
    \end{align*}
    Then
    \begin{align*}
        \sum_{i}t_{ii_1}\big(\sum_{k} F_{ij}^k \sum_{\tilde{k}}\hat{t}_{\tilde{k}k}d_{\tilde{k}}\big)
        &=\sum_{i}t_{ii_1} \big(\sum_{\tilde{i}}\hat{t}_{\tilde{i}i}d_{\tilde{i}}\cdot \sum_{\tilde{j}}\hat{t}_{\tilde{j}j}d_{\tilde{j}}\big)
        \\&=\sum_{\tilde{i}}\delta_{\tilde{i}i_1}d_{\tilde{i}}\cdot \sum_{\tilde{j}}\hat{t}_{\tilde{j}j}d_{\tilde{j}}
        \\&=d_{{i_1}}\cdot \sum_{\tilde{j}}\hat{t}_{\tilde{j}j}d_{\tilde{j}}.
    \end{align*}
    Using the same technique again, we have \[
    d_{i_1}\cdot d_{j_1}=\sum_{\tilde{k}}\sum_{i,j,k} t_{ii_1}t_{jj_1}\hat{t}_{\tilde{k}k} F_{ij}^k d_{\tilde{k}}.
    \]Hence, with $i_1,j_1$ replaced by $\tilde{i},\tilde{j}$, we obtain $\widetilde{F}_{\tilde{i}\tilde{j}}^{\tilde{k}}=\sum_{i,j,k}t_{i\tilde{i}}t_{j\tilde{j}}\hat{t}_{\tilde{k}k}F_{ij}^k.$

    As for the comultiplication, the above technique is still available and we omit the details here.
\end{proof}


 The following result shows that, under the condition that the field is algebraically closed, we can simplify the situation to the case presented in Theorem \ref{thm1} by employing a base-change strategy.

\begin{lemma}
\label{FequalG}
    Let $A$ be a symmetrically self-dual Hopf algebra over an algebraically closed field $\bfk$. Then there exists a basis $\{c_i\}_{i\in I}$ of $A$ with 
$F_{ij}^k=G_{ij}^k$ for any $i,j,k\in I$.
\end{lemma}


\begin{proof}
By assumption, $A$ is symmetrically self-dual implies 
there exists a Hopf algebra isomorphism 
$\varphi:A\to A^*$ with $\varphi=\varphi^*$.
 Let $\{d_{\tilde{i}}\}_{\tilde{i}\in I}$ be a basis of $A$ and $\{d^*_{\tilde{i}}\}_{\tilde{i}\in I}$ be its dual basis. Let  $S=(s_{\tilde{i}\tilde{j}})$ be the matrix representation of $\varphi$ with respect to these two bases, namely, $\varphi(d_{\tilde{i}})=\sum_{\tilde{j}}s_{\tilde{j}\tilde{i}}d^*_{\tilde{j}}$.
Note that $\varphi=\varphi^*$ implies $S$ is symmetric. Since the field is algebraically closed,
any symmetric matrix 
can always be diagonalized by the spectral decomposition theorem.
Consequently, $S=T^t\cdot T$ for some matrix $T=(t_{i\tilde{i}})$. 

Denote by ${T}^{-1}:=(\hat{t}_{\tilde{i}i})$ the inverse matrix of $T$, and let $c_{{i}}=\sum_{\tilde{i}}\hat{t}_{\tilde{i}i}d_{\tilde{i}}$ for any $i\in I$. Then $\{c_i\}_{i\in I}$ is a basis of $A$. Denote its dual basis by $\{c^*_i\}_{i\in I}$.
We have $c^*_j(c_i)=c^*_j(\sum_{\tilde{i}}\hat{t}_{\tilde{i}i}d_{\tilde{i}})=\delta_{ij}$. Therefore,  $c^*_j(d_{\tilde{i}})=\sum_{{i}}{t}_{i\tilde{i}}\delta_{ij}={t}_{j\tilde{i}}$, which implies that $c^*_j=\sum_{\tilde{i}}t_{j\tilde{i}}d^*_{\tilde{i}}$. In other words, $d^*_{{\tilde{i}}}=\sum_{{i}}\hat{t}_{\tilde{i}i}c^*_{{i}}$ and similarly, $d_{\tilde{i}}=\sum_{{i}}{t}_{i\tilde{i}}c_{{i}}$. 

Now let us focus on the self-dual $\varphi$. We have \[
    \varphi(\sum_{{i}}{t}_{i\tilde{i}}c_{{i}})=\sum_{\tilde{j}}s_{\tilde{j}\tilde{i}}\sum_{{j}}\hat{t}_{\tilde{j}j}c^*_{{j}}.
    \]
    It follows that \begin{align*}   \varphi(c_{{i}})&=\sum_{{\tilde{i}}}\hat{t}_{\tilde{i}i}\sum_{\tilde{j}}s_{\tilde{j}\tilde{i}}\sum_{{j}}\hat{t}_{\tilde{j}j}c^*_{{j}}    \\&=\sum_{{\tilde{i}}}\hat{t}_{\tilde{i}i}\sum_{\tilde{j}}\sum_kt_{k\tilde{i}}t_{k\tilde{j}}\sum_{{j}}\hat{t}_{\tilde{j}j}c^*_{{j}}
    \\&=\sum_{j,k}\delta_{ik}\delta_{jk}c^*_j=\sum_{j}\delta_{ij}c^*_j=c^*_i.
    \end{align*}
    Note that $$c_i^*\cdot c_j^*=\varphi(c_i)\cdot \varphi(c_j)=\varphi(\sum_{k} F_{ij}^k c_k)=\sum_{k} F_{ij}^k c_k^*.$$ Compared with that $c_i^*\cdot c_j^*=\sum_k G_{ij}^k c_k^*$, we immediately obtain $F_{ij}^k=G_{ij}^k$.
\end{proof}


    

\subsection{$\imath$Hopf algebras} Now we can state the main result of this paper.

\begin{theorem}
\label{thm2}

 Let $A$ be a symmetrically self-dual Hopf algebra with the basis $\{d_{\tilde{i}}\}_{\tilde{i}\in I}$, and the self-dual $\varphi$ is given by $\varphi(d_{\tilde{i}})=\sum_{\tilde{j}}s_{\tilde{j}\tilde{i}}d^*_{\tilde{j}}$ for any  $\tilde{i}\in I$. 
 Let $(A^\imath,\diamond)$ be a vector space with the same basis as $A$, equipped with the multiplication $d_{\tilde{i}}\diamond d_{\tilde{j}}=\sum_{\tilde{k}\in I} {}^\imath \widetilde{F}^{\tilde{k}}_{\tilde{i}\tilde{j}} d_{\tilde{k}}$, where
   \begin{align}
       \label{formula of iFijk}
 {}^\imath \widetilde{F}^{\tilde{k}}_{\tilde{i}\tilde{j}}=\sum_{y_1,y_2,x,z}
    s_{y_1y_2}\widetilde{G}_{y_1x}^{\tilde{i}}\widetilde{G}_{zy_2}^{\tilde{j}}\widetilde{F}_{xz}^{\tilde{k}}  \qquad\text{  for any  }\tilde{i},\tilde{j},\tilde{k}\in I.
   \end{align}
   Then $(A^\imath,\diamond)$ is an associative algebra with the unit $u(1)$.
   
\end{theorem}

\begin{proof}
Firstly, we assume $\bfk$ is an algebraically closed field. 
In this case, according to Lemma \ref{FequalG}, the representation matrix $S=(s_{ij})$ of $\varphi$ is symmetric and has a decomposition $S=T^t\cdot T$ for some matrix $T=(t_{i\tilde{i}})$, moreover, there exists a basis $\{c_{{i}}=\sum_{\tilde{i}}\hat{t}_{\tilde{i}i}d_{\tilde{i}}\}_{i\in I}$ of $A^\imath$ 
 with 
$F_{ij}^k=G_{ij}^k$ for any $i,j,k\in I$. 


    By Theorem \ref{thm1}, there exists an algebra structure of $A^\imath$ with the unit $u(1)$ and multiplication $c_i\diamond c_j=\sum_{k\in I} {}^\imath F^k_{ij} c_k$, where\[
   {}^\imath F^k_{ij}=\sum_{x,y,z}F^i_{yx}F^j_{zy}  F^k_{xz}.
    \]
    According to the base-change formula \eqref{trans1}, 
    we have
    \[ {}^\imath\widetilde{F}_{\tilde{i}\tilde{j}}^{\tilde{k}}=\sum_{i,j,k}t_{i\tilde{i}}t_{j\tilde{j}}\hat{t}_{\tilde{k}k}{}^\imath F_{ij}^k=\sum_{i,j,k}t_{i\tilde{i}}t_{j\tilde{j}}\hat{t}_{\tilde{k}k}\sum_{x,y,z}F^i_{yx}F^j_{zy}  F^k_{xz}.
    \]
    It remains to show that this formula coincides with \eqref{formula of iFijk}. 
Note that $F_{ij}^k=G_{ij}^k$ for any $i,j,k\in I$. Then by applying \eqref{trans1} again, we have
    \begin{align*}
        {}^\imath\widetilde{F}_{\tilde{i}\tilde{j}}^{\tilde{k}}
&=\sum_{i,j,k}t_{i\tilde{i}}t_{j\tilde{j}}\hat{t}_{\tilde{k}k}\sum_{x,y,z}G^i_{yx}G^j_{zy}  F^k_{xz}
\\&=\sum_{i,j,k}t_{i\tilde{i}}t_{j\tilde{j}}\hat{t}_{\tilde{k}k}
\sum_{x,y,z}\sum_{x_1,y_1,i_1}t_{yy_1}t_{xx_1}\hat{t}_{i_1i}\widetilde{G}_{y_1x_1}^{i_1}
\\&\qquad\cdot\sum_{y_2,z_2,j_2}t_{zz_2}t_{yy_2}\hat{t}_{j_2j}\widetilde{G}_{z_2y_2}^{j_2}
\sum_{x_3,z_3,k_3}\hat{t}_{x_3x}\hat{t}_{z_3z}t_{kk_3}\widetilde{F}_{x_3z_3}^{k_3}
\\&=\sum_{x_1,y_1,i_1}\sum_{y_2,z_2,j_2}\sum_{x_3,z_3,k_3}\sum_{i}t_{i\tilde{i}}\hat{t}_{i_1i}
\sum_{j}t_{j\tilde{j}}\hat{t}_{j_2j}\sum_{k}\hat{t}_{\tilde{k}k}t_{kk_3}
\\&\qquad\cdot\sum_{y}t_{yy_1}t_{yy_2}\sum_{x}t_{xx_1}\hat{t}_{x_3x}
\sum_{z}t_{zz_2}\hat{t}_{z_3z}\widetilde{G}_{y_1x_1}^{i_1}\widetilde{G}_{z_2y_2}^{j_2}\widetilde{F}_{x_3z_3}^{k_3}
\\&=\sum_{x_1,y_1,i_1}\sum_{y_2,z_2,j_2}\sum_{x_3,z_3,k_3}\delta_{i_1\tilde{i}}\delta_{j_2\tilde{j}}\delta_{k_3\tilde{k}}\delta_{x_1x_3}\delta_{z_2z_3}
    \sum_{y}t_{yy_1}t_{yy_2}\widetilde{G}_{y_1x_1}^{i_1}\widetilde{G}_{z_2y_2}^{j_2}\widetilde{F}_{x_3z_3}^{k_3}
    \\&=\sum_{y_1,y_2,x_3,z_3}
    \sum_{y}t_{yy_1}t_{yy_2}\widetilde{G}_{y_1x_3}^{\tilde{i}}\widetilde{G}_{z_3y_2}^{\tilde{j}}\widetilde{F}_{x_3z_3}^{\tilde{k}}
    \\&=\sum_{y_1,y_2,x_3,z_3}
    s_{y_1y_2}\widetilde{G}_{y_1x_3}^{\tilde{i}}\widetilde{G}_{z_3y_2}^{\tilde{j}}\widetilde{F}_{x_3z_3}^{\tilde{k}}.
    \end{align*}
Therefore, we obtain that $(A^\imath,\diamond)$ is an associative algebra with the unit $u(1)$ when  $\bfk$ is algebraically closed. 

In the general case, let $\bfk$ be an arbitrary field and $\overline{\bfk}$ be its algebraic closure. In this case, we can establish the algebraic structure of $(A^\imath,\diamond)$ over $\overline{\bfk}$. Note that ${}^\imath \widetilde{F}^{\tilde{k}}_{\tilde{i}\tilde{j}}\in\bfk$ for all $\tilde{i},\tilde{j},\tilde{k}\in I$, hence $(A^\imath,\diamond)$ is an associated algebra over $\bfk$ as well. 
    
    This completes the proof.
\end{proof}

We call $(A,\diamond)$ the \textit{$\imath$Hopf algebra} associated with the self-dual Hopf algebra $A$.

\begin{remark}

$(1)$ The simple cases $F_{ij}^k=G_{ij}^k$ and $F_{ij}^k=G_{ij}^k\cdot\dfrac{a_k}{a_ia_j}$ are special cases of Theorem \ref{thm2}, corresponding to that the representation matrix $S$ is a unit matrix or a diagonal matrix, respectively.

$(2)$ For general self-dual Hopf algebras, we can also construct the associated
$\imath$Hopf algebras. The construction should be framed within the context of Hopf pairings, and we will address this in a subsequent paper.
\end{remark}

\section{Applications}


It is well known that Taft algebras are a family of important finite-dimensional self-dual Hopf algebras. We apply our $\imath$Hopf algebra approach to Taft algebras as a typical example. For the reader's convenience, we recall the definition of Taft algebras and give the explicit self-dual morphism in the case of dimension 4.

Let $n\ge2$ be a fixed integer. Assume that the field $\mathbf{k}$ contains an $n$-th primitive root of unity,  denoted by $q$, and that $n$ is not divisible by the characteristic of $\mathbf{k}$.
The Taft algebra $H_n(q)$ is generated by two elements $g$ and $h$ subject to the relations\[
g^n=1,\quad h^n=0,\quad hg=qgh.
\]
It becomes a Hopf algebra with coalgebra structure $\Delta$ and antipode $S$ given by
\begin{align*}
&\Delta(g)=g\otimes g,\quad \Delta(h)=1\otimes h+h\otimes g,\quad \varepsilon(g)=1,\\
&\varepsilon(h)=0,\quad S(g)=g^{-1}=g^{n-1},\quad S(h)=-q^{-1}g^{n-1}h.
\end{align*}
Note that $\dim H_n(q)=n^2$ and $\{g^ih^j\mid 0\le i,j\le n-1\}$ forms a $\mathbf{k}$-basis for $H_n(q)$.


\begin{proposition}
    The Hopf algebra $H_2(q)$ is a symmetrically self-dual Hopf algebra.
\end{proposition}

\begin{proof}

For convenience, we describe the multiplication of $H_2(q)$ by the following table:
\[\begin{tabular}{|c|cccc|}
		\hline
&		$1$ & $g$ & $h$ & $gh$ \\\hline
	$1$&	$1$& $g$ & $h$ &$gh$ \\
$g$&		$g$& $1$ & $gh$ &$h$ \\
$h$&		$h$& $-gh$ & $0$ &$0$ \\
$gh$&		$gh$& $-h$ & $0$ &$0$ \\
		\hline
		\end{tabular}\]	
The comultiplications of the basis $\{1,g,h,gh\}$ are 
\begin{align*}
\Delta(1)=1\otimes 1,\quad \Delta(g)=g\otimes g,\quad \Delta(h)=1\otimes h+h\otimes g,\quad \Delta(gh)=g\otimes gh+gh\otimes 1.
\end{align*}
Assume $H^*_2(q)$ is the dual space of $H_2(q)$ with the dual basis $\{1^*,g^*,h^*,(gh)^*\}$. Then through the duality of comultiplication of $H_2(q)$, we can describe the multiplication of $H^*_2(q)$:
\[\begin{tabular}{|c|cccc|}
		\hline
&		$1^*$ & $g^*$ & $h^*$ & $(gh)^*$ \\\hline
	$1^*$&	$1^*$&0  & $h^*$ &0 \\
$g^*$&	0	& $g^*$ & 0 &$(gh)^*$ \\
$h^*$&	0	& $h^*$ &0  &0 \\
$(gh)^*$&		$(gh)^*$& 0 &0  &0 \\
		\hline
		\end{tabular}\]	

Let $S=(s_{ij})_{4\times 4}$ be the following matrix $$\begin{pmatrix}
        1&1  &  & \\
		1& -1 &  & \\
		&  &1  &1 \\
		&  &  1& -1
    \end{pmatrix}.$$
It is easy to check that the assignment $$\varphi(1,g,h, gh)=(1^*, g^*,h^*, (gh)^*)S$$ defines an algebra isomorphism $$\varphi:H_2(q)\longrightarrow H^*_2(q).$$
By Proposition \ref{bialgiso}, $\varphi$ is also an isomorphism of bialgebras. Moreover, the observation $S$ is symmetric implies $\varphi=\varphi^*$. Then the result follows.
\end{proof}

According to Theorem \ref{thm2}, we can associate to $H_2(q)$ an $\imath$Hopf algebra structure, called \textit{$\imath$Taft algebras} and denoted by $H^\imath_2(q)$.

Now we deduce the new multiplication ``$\diamond$'' in the $\imath$Taft algebras $H^\imath_2(q)$. For convenience, denote by $(d_1,d_2,d_3,d_4)=(1,g,h,gh)$. We calculate $d_2\diamond d_3$ and $d_3\diamond d_4$ in detail as follows. 
Since $$\Delta(d_2)=g\otimes g=d_2\otimes d_2, \quad \Delta(d_3)=1\otimes h+h\otimes g=d_1\otimes d_3+d_3\otimes d_2,$$ and $$\Delta(d_4)=g\otimes gh+gh\otimes 1=d_2\otimes d_4+d_4\otimes d_1,$$ we obtain
\begin{align*}
      d_2\diamond d_3=s_{23}d_2\cdot d_1+s_{22}d_2\cdot d_3=-gh=-d_4;\\
\end{align*}
and
\begin{align*}
      d_3\diamond d_4=&s_{14}d_3\cdot d_2+s_{11}d_3\cdot d_4+s_{34}d_2\cdot d_2+s_{13}d_2\cdot d_4\\
      =&d_3\cdot d_4+d_2\cdot d_2=h\cdot gh+g\cdot g=1=d_1.
\end{align*}
One can calculate $d_i\diamond d_j$ for any $1\leq i,j\leq 4$ in a similar manner. We omit the details here and simply conclude the multiplication of $H^\imath_2(q)$ as follows:
\[\begin{tabular}{|c|cccc|}
		\hline
$\diamond$ &		$d_1$ & $d_2$ & $d_3$ & $d_4$ \\\hline
	$d_1$&	$d_1$& $d_2$ & $d_3$ &$d_4$ \\
$d_2$&		$d_2$& $-d_1$ & $-d_4$ &$d_3$ \\
$d_3$&		$d_3$& $-d_4$ & $d_2$ &$d_1$ \\
$d_4$&		$d_4$& $d_3$ & $d_1$ &$-d_2$ \\
		\hline
		\end{tabular}\]	

 Consequently, we have
 \begin{proposition}
    The $\imath$Hopf algebra $H_2^\imath(q)$ is commutative, and it is isomorphic to the group algebra of $\mathbb{Z}/4\mathbb{Z}$ if the base field $\bfk$ is algebraically closed.
 \end{proposition}

 \begin{proof}
    Obviously, $H_2^\imath(q)$ is commutative. Since $\bfk$ is algebraically closed, there exists an element $\sqq\in\bfk$ with $\sqq^4=-1$. Then we have \[
     (\sqq d_3)^2=\sqq^2 d_2,\quad (\sqq d_3)^3=-\sqq^3 d_4,\quad (\sqq d_3)^4= d_1=1. 
     \]It immediately follows that $H_2^\imath(q)\cong \bfk \langle \sqq d_3\rangle\cong\bfk (\mathbb{Z}/4\mathbb{Z})$.
 \end{proof}

\noindent{\bf Acknowledgment.}
The authors are grateful to Fan Qin for proposing the research topic of $\imath$Hopf algebras.
This work was partially supported by the Natural Science
Foundation of Xiamen (Grant No. 3502Z20227184),
Fujian Provincial Natural Science Foundation of China (Grant Nos. 2024J010006 and 2022J01034),
the National Natural Science Foundation of China (Grant No. 12271448), and the Fundamental
Research Funds for Central Universities of China (Grant No. 20720220043).



\begin{thebibliography}{99}




\bibitem{AL08}  M. Ahmed and F. Li,
{\em Self-dual weak Hopf algebras}, Acta Math. Sin. (Engl. Ser.) {\bf24} (12) (2008), 1935--1948.

\bibitem{CLinR} J. Chen, Y. Lin and S. Ruan,
{\em Realization of $\imath$quantum groups via $\Delta$-Hall algebras
}, J. Algebra {\bf 653} (2024), 378--403.


\bibitem{CLR} J. Chen, M. Lu and S. Ruan,
{\em $\imath$Quantum groups of split type via derived Hall algebras}, J. Algebra {\bf 610} (2022), 391--408.

\bibitem{DR13}  A. Davydov and I. Runkel,
{\em $\mathbb{Z}/2\mathbb{Z}$-extensions of Hopf algebra module categories by their base categories}, Adv. Math. {\bf 247} (2013), 192--265.

\bibitem{Fo}  L. Foissy, {\em Plane posets, special posets, and permutations}, Adv. Math. {\bf240} (2013), 24--60.

\bibitem{G}  J. Green, {\em Hall algebras, hereditary algebras and quantum groups}, Invent. Math. {\bf 120} (2) (1995), 361--377.

\bibitem{GM00} E. L. Green and E. N. Marcos,
{\em Self-dual Hopf algebras}, Comm. Algebra {\bf28} (6) (2000), 2735--2744.

\bibitem{HLY} H. Huang, L. Li and Y. Ye,
{\em Self-dual Hopf quivers}, Comm. Algebra {\bf33} (12) (2005), 4505--4514.


\bibitem{LW20a} M. Lu and W. Wang,
{\em Hall algebras and quantum symmetric pairs of Kac-Moody type}, Adv. Math. {\bf430} (2023), 109215.

\bibitem{LW19a} M. Lu and W. Wang,
{\em Hall algebras and quantum symmetric pairs I: foundations}, Proc. London Math. Soc. {\bf 124} (3) (2022), no. 1, 1--82.

\bibitem{MR11} C. Malvenuto and C. Reutenauer,
{\em A self paired Hopf algebra on double posets and a Littlewood-Richardson rule}, J. Combin. Theory Ser. A {\bf118} (4) (2011), 1322--1333.

\bibitem{R1} C.M. Ringel, {\em Hall algebras and quantum groups}. Invent. Math. {\bf 101} (1990), 583--592.

\bibitem{T71} Earl J. Taft, {\em The order of the antipode of finite-dimensional Hopf algebra}, Proc. Nat. Acad. Sci. U.S.A. {\bf68} (1971), 2631--2633.

\bibitem{X97} J. Xiao, {\em Drinfeld double and Ringel-Green theory of Hall algebras}, J. Algebra {\bf190} (1997), 100--144.

\end{thebibliography}
\end{document}